\documentclass[11pt]{article}
\usepackage[T2A]{fontenc}
\usepackage[cp1251]{inputenc}
\usepackage[english]{babel}
\usepackage{amsmath,amssymb,amsfonts,amsthm}
\usepackage{manyfoot}

\usepackage{anyfontsize}
\renewenvironment{proof}
{\textbf{\textit{Proof.}} }
{\hfill$\square$\smallskip}
\theoremstyle{plain}

\newtheorem{theorem}{\indent {Theorem}}[section]{}
\newtheorem{proposition}[theorem]{\indent {Proposition}}
\newtheorem{lemma}[theorem]{\indent {Lemma}}
\newtheorem{corollary}[theorem]{\indent {Corollary}}
\theoremstyle{definition}
\newtheorem{remark}[theorem]{\indent {Remark}}

\usepackage{titlesec}
\titlelabel{\thetitle. \,}

\counterwithin{equation}{section}

\usepackage{authblk}

\title{\textbf{Rings on quotient divisible abelian groups}}
\author{Kompantseva E., Nguyen T.~Q.~T}

\begin{document}
	\maketitle

\begin{abstract}
	The paper is devoted to the study of absolute ideals of groups in the class $\mathcal{QD}1$, which consists of all quotient divisible abelian groups of torsion-free rank 1. A ring is called an $AI$-ring (respectively, an $RF$-ring) if it has no ideals except absolute ideals (respectively, fully invariant subgroups) of its additive group. An abelian group is called an $RAI$-group (respectively, an $RFI$-group) if there exists at least one $AI$-ring (respectively, $FI$-ring) on it. If every absolute ideal of an abelian group is a fully invariant subgroup, then this group is called an $afi$-group.
	It is shown that every group in $\mathcal{QD}1$ is an $RAI$-group, an $RFI$-group, and an $afi$-group. Thus, Problem 93 of L. Fuchs' monograph \emph{``Infinite Abelian Groups, Vol. II, New York-London: Academic Press, 1973''} is resolved within the class $\mathcal{QD}1$. For any group in $\mathcal{QD}1$, all rings on it that are $AI$-rings are described. Furthermore, the set of all $AI$-rings on $G \in \mathcal{QD}1$ coincides with the set of all $FI$-rings on $G$. In addition, the principal absolute ideals of groups in $\mathcal{QD}1$ are described.\\
	
	
\end{abstract}

\section{Introduction}	
Apart from vector spaces, abelian groups are certainly most commonly found in rings and fields. A ring on an abelian group $G$ is a ring whose additive group coincides with $G$. 
The first papers, in which the relations between the properties of a ring and the structure of its additive group were investigated, provided only a superficial analysis in very special cases \cite{B48,B-Z51,F1956,R-Sz50,Sz49}.
They stimulated interest in the additive groups of rings, and several more substantial papers were published  in the next decade. These papers initiated a systematic study of rings on groups, which has currently become an independent research branch of Abelian group theory (see \cite{A-K-N13,A-A-M24,A24,A-B-R22,B-L74,B-P61,Fe83-88,F2015,Ja82,H-K-N-S16,K10,K-T25} and others).

All groups considered in this work are abelian, and the word ``group'' means an ``abelian group'' everywhere in what follows.

When dealing with rings on a given group $G$, an inevitable problem is to study those subgroups of $G$ that have certain property $\mathcal{P}$ in any ring on $G$. L. Fuchs in \cite{F2015} calls them absolute $\mathcal{P}$. For example, subgroups of a group $G$ which are subrings \cite{A-K-N13}, ideals \cite{B-L74,Fe83-88,Fried71,F2015,Ja82,K-Ph19,K-T22b}, nil-ideals \cite{Fe83-88,F2015,K10,K14}, quasi-regular ideals (consequently, they are contained in the Jacobson radical)\cite{Fe83-88,F12,F2015,Ja82,K10,K-N-V24}, annihilators \cite{Fe83-88,F1956,F12}, etc., in every ring on $G$ are studied. 
We will consider absolute ideals, i.e. subgroups that are (one- or two-sided) ideals in every ring on a given group. 
It is easy to see that the problems of left, right, or two-sided ideals are identical since an anti-isomorphic ring is defined on the same group. In \cite{Fried64}, an ideal $F$ of the endomorphism ring $E(G)$ of the group $G$ is defined, and it is shown that a subgroup $A$ is an absolute ideal of $G$ if and only if $A$ is invariant with respect to this ideal, i.e. $F(A) \subseteq A$. Therefore, any fully invariant subgroup of the group $G$ is its absolute ideal; however, the converse statement is not true. In \cite{Fried64} E. Fried formulated the problem of describing groups, in which every absolute ideal is a fully invariant subgroup; such groups are called $afi$-groups. $Afi$-groups in the class of fully transitive $p$-groups are described in \cite{McL75a,McL75b}, mixed $afi$-groups were studied in \cite{K-Ph19}, and torsion-free $afi$-groups were considered in \cite{K-T22b,K-T24}.


Other problems related to absolute ideals of groups consist in the description of $RAI$-groups and $RFI$-groups. A ring $R$ is called an $AI$-ring (respectively, an $FI$-ring) if any ideal of $R$ is an absolute ideal (respectively, a fully invariant subgroup) of the additive group of $R$. A group on which there exists at least an $AI$-ring (respectively, an $FI$-ring) is called an $RAI$-group (respectively, an $RFI$-group). The problem of describing $RAI$-groups was formulated by L. Fuchs in  \cite[Problem 93]{F12}, such groups were studied in \cite{Fried71,K-Ph19,K-T22b,K-T25,McL75b}. The problem of describing $RFI$-groups was posed in \cite[Problem 66]{F1966}, and later K. McLean described $RFI$-groups in the class of all $p$-groups in \cite{McL75a}. 

Our paper is devoted to the study of the problem related to absolute ideals in the class of quotient divisible groups of torsion-free rank 1. 
A group $G$ is called to be quotient divisible if it does not contain nonzero divisible torsion subgroups but contains a free finite-rank subgroup $F$ such that $G/F$ is a divisible torsion group. The basis of the free group $F$ is called the basis of the quotient divisible group $G$. The concept of a quotient divisible group was introduced by R. Beaumont and R. Pierce in \cite{B-P61} to describe torsion-free groups admitting a ring structure that is embedded in a semisimple separable algebra. Later, this concept was extended to the case of mixed groups in \cite{F-W98}. Currently, the theory of quotient divisible groups attracts many algebraists \cite{A-B-V-W07,A-W2004,A-B-W07,B-P61,B2010,B-S2012,D2007,F-W2003,Fo2014,F-W98,K-N23,W2003}. Let $\mathcal{QD}1$ denote the class of all quotient divisible groups of rank 1.

This paper is a continuation of the papers \cite{K-N23} and \cite{K-N-V24}, where authors respectively study the group $\text{\normalfont Mult}\,G$ of all multiplications and radicals of rings on groups $G\in\mathcal{QD}1$. 
In Section \ref{sec2}, we describe principal ideals of an arbitrary ring on a group $G \in \mathcal{QD}1$ (Theorem \ref{the2.3}). This result allows us in Section \ref{sec3} to describe principal absolute ideals of groups in $\mathcal{QD}1$ (Theorem \ref{the3.3}). The principal absolute ideal of a group $G$ generated by an element $g\in G$ is the smallest absolute ideal $(g)_{AI}$ of the group $G$ containing $g$. Since each absolute ideal of a group is the sum of principal absolute ideals, many questions related to absolute ideals are reduced to the case of principal absolute ideals (for example, see \cite[Lemma 3.1]{K-T25}). 
In Corollary \ref{cor3.2}, we show that any quotient divisible group of rank 1 is an $afi$-group and an $RFI$-group, therefore, it is an $RAI$-group. However, this statement does not clarify which rings on groups in $\mathcal{QD}1$ are $AI$-rings and $FI$-rings. 
In Theorem \ref{the3.4}, for an arbitrary group $G \in \mathcal{QD}1$ we describe all rings on it that are $AI$-rings. Moreover, we show that any $AI$-ring on $G \in \mathcal{QD}1$ is also an $FI$-ring.

Unless otherwise stated, for all definitions and notations, we refer to \cite{F12,F2015,K-N23}.

\section{Principal ideals of rings on quotient divisible groups}\label{sec2}
The aim of this section is to describe the principal ideals of rings on groups in $\mathcal{QD}1$. As usual, $\mathbb{N},\;P$ are sets of natural numbers and all prime numbers, respectively, $\mathbb{Z}$ is the ring of integers, $\mathbb{Q}$ is the group of rational numbers, $\widehat{\mathbb{Z}}_p$ is the ring of $p$-adic integers, $\mathbb{Z}_n$ is a cyclic group of order $n$. If $R$ is a unital ring, then $Re$ is a cyclic module over $R$ generated by the element $e$. If $G$ is a group, $p\in P$, $A\subseteq G$, then $T(G)$ is the torsion part of $G$, $T_p(G)$ is a $p$-primary component of $G$, $\langle A\rangle_\ast $ is a pure subgroup of $G$ generated by the set $A$ \cite[Chapter 5, Section 1]{F2015}. If $g\in G$, then the order and the $p$-height of the element $g$ are denoted by $o(g)$ and $h_p(g)$, respectively.

Let us recall the basic concepts. A function $\chi$ on the set $P$ with values in the set $\{\infty, 0, 1, 2, \ldots\}$ is called a characteristic (see \cite[Chapter 12, Section 1]{F2015}). The characteristic $\chi$ will be written in the form $\chi = (k_p)_{p \in P}$, here $\chi(p) = k_p$.
Two characteristics $(k_p)_{p \in P}$ and $(m_p)_{p \in P}$ are equivalent if the set $S=\{p~|~k_p\neq m_p\}$ is finite, and also $k_p<\infty$ and $m_p<\infty$ for all $p\in S$. Equivalence classes of characteristics are called types. If a type contains a characteristic consisting of $0$'s, then it is called the zero type. A type containing an idempotent characteristic, i.e. a characteristic $(k_p)_{p \in P}$ such that $k_p$ is either $0$ or $\infty$ for every prime $p$, is called an idempotent type.

According to \cite{D2007}, every group in $\mathcal{QD}1$ is uniquely determined, up to isomorphism, by its cocharacteristic $\text{\normalfont cochar}\,G$. Moreover, for any characteristic $\chi$ there exists a group $G\in\mathcal{QD}1$ with $\text{\normalfont cochar}\,G=\chi$.
Let $\chi$ be a characteristic. If $\chi$ belongs to a non-zero type, then we consider the direct product
\begin{equation}\label{eqq2.1}
	\mathbb{Z}_\chi = \prod\limits_{p\in P}\widehat{\mathbb{Z}}_pe_p
\end{equation} 
of cyclic $p$-adic modules $\widehat{\mathbb{Z}}_pe_p$ such that $o(e_p)=p^{\chi(p)}$ for all $p\in P$ (we set $p^\infty=\infty$). If $o(e_p)<\infty$, then the module $\widehat{\mathbb{Z}}_pe_p$ coincides with $\mathbb{Z}e_p$. The quotient divisible group $G$ of rank 1 with 
$\text{\normalfont cochar}\,G=\chi$
is of the form
\begin{equation}\label{eqq2.2}
	G=\langle  e,T(\mathbb{Z}_\chi)\rangle_\ast,
\end{equation}
where $e=(e_p)_{p\in P}$. We denote $P_\chi =\{p\in  P  ~|~ \chi(p)\neq 0\}$. The system $\{e\}$ is a basis of the quotient divisible group $G$ \cite[Theorem 4]{D2007}, while the system $\{e_p~|~p\in P_\chi\}$ satisfying the conditions \eqref{eqq2.1} and \eqref{eqq2.2} is called a $\Pi$-basis of $G$ \cite{K-N23}. 

If $\chi$ belongs to the zero type and $m=\prod\limits_{\chi(p)\neq 0}p^{\chi(p)}$, then the group $G$ in $\mathcal{QD}1$ with $\text{\normalfont cochar}\,G=\chi$  is of the form  $G=\mathbb{Q}\oplus \mathbb{Z}_m$. Therefore, the group $G\in\mathcal{QD}1$ is reduced if and only if  $\text{\normalfont cochar}\,G$ does not belong to the zero type. Let us denote by $\mathcal{RQD}1$ the class of all reduced quotient divisible groups of rank 1.

Let $G\in\mathcal{RQD}1$, $\text{\normalfont cochar}\,G=\chi$, and let $E=\{e_p~|~p\in P_\chi\}$ be a  $\Pi$-basis of the group $G$, $e=(e_p)_{p\in P_\chi}$. We denote  $P_\infty(\chi)=\{p\in  P  ~|~ \chi(p)=\infty\}$, $P_{_N}(\chi)=P_\chi \setminus P_\infty(\chi)=\{p\in  P  ~|~ \chi(p)\in\mathbb{N}\}$. If $P_1\subseteq P$, then a $P_1$-integer is a nonzero integer such that any its prime divisor (if it exists) is contained in $P_1$, and a $P_1$-fraction is a rational number, which can be represented in the form of a fraction whose numerator and denominator are $P_1$-integers. If $p\in P$, $P_1\subseteq P$, then $\pi_ p , \pi_{P_1}$ denote the projections of the group $\mathbb{Z}_\chi$ onto subgroups $\widehat{\mathbb{Z}}_pe_p$ and $\prod\limits_{p\in P_1}\widehat{\mathbb{Z}}_pe_p$, respectively. Note that if $P_1$ is a finite subset of $P_{_N}(\chi)$, then $\pi_{P_1}(G)\subseteq G$ and $\pi_{P\setminus P_1}(G)\subseteq G$.
Let $g\in G$. In \cite{K-N23}, the number $c(g)$ is defined as follows
\begin{equation*}
	c(g)=\left\{\begin{array}{cl}
		\prod\limits_{p\in P_\infty(\chi)}p^{h_p(g)}, & \text{if} \; g\notin T(G), P_\infty(\chi)\neq \varnothing\\
		1, & \text{if} \; g\notin T(G), P_\infty(\chi)= \varnothing\\
		0, & \text{if}\; g\in T(G),
	\end{array}\right.
\end{equation*}
and it is also proved that there exists a set $P_0\subseteq P_\chi$ such that 
\begin{equation}\label{eqq2.5}
	P'=P_\chi\setminus P_0\;\text{is a finite subset of the set}\; P_{_N}(\chi),
\end{equation}
and the element $g$ can be written as follows
\begin{equation}\label{eqq2.6}
	g=c(g)re_0+t,
\end{equation}
where $e_0=\pi_{P_0}(e_0)$, $r$ is a $P\setminus P_0$-fraction, $t\in\bigoplus\limits_{p\in P'}\widehat{\mathbb{Z}}_pe_p$.
The set $P_0=P_0(g)$ satisfying the conditions \eqref{eqq2.5} and \eqref{eqq2.6} is called a $g$-defining set with respect to the $\Pi$-basis $E$. Note that the set $P_0(g)$ is not uniquely defined.

Let $G$ be a group. Recall that the characteristic of the element $g\in G$ is the characteristic $\text{\normalfont char}\,g$ defined by $[\text{\normalfont char}\,g](p)=h_p(g)$. For any group $G$ and any characteristic $\eta$ we denote $G(\eta)=\{x\in G~|~\text{\normalfont char}\,x\geq \eta\}$. It is easy to see that $G(\eta)$ is a fully invariant subgroup of the group $G$ (for example, see \cite[Chapter 12, Section 1]{F2015}). 

\begin{remark}\label{rem2.1}
	If $G\in\mathcal{RQD}1$, $g\in G$, then the group $G(\text{\normalfont char}\,g)$ can be written in the form
	\begin{equation*}
		G(\text{\normalfont char}\,g)=c_gG_0\bigoplus \bigoplus\limits_{p\in P'}p^{h_p(g)}T_p(G)=c_gG_0+ \bigoplus\limits_{p\in P}p^{h_p(g)}T_p(G),
	\end{equation*}
	where $c_g=c(g)$, $P_0$ is any $g$-defining set, $G_0=\pi_{P_0}(G)$, $P'=P \setminus P_0$. \hfill$\square$\smallskip
\end{remark}

To describe the group $ \bigoplus\limits_{p\in P}p^{h_p(g)}T_p(G)$ in the case $T_p(G)$ are cyclic groups for all $p\in P$ (for example, if $G\in\mathcal{QD}1$), we prove the following lemma. Note that if $T_p(G)$ is a nonzero cyclic group, then  $p^\infty T_p(G)=0$.

\begin{lemma}\label{lem2.1}
	Let $G$ be a group, $T(G)=\bigoplus\limits_{p\in P}\mathbb{Z}e_p$, where $o(e_p)=p^{\alpha_p},\;\alpha_p\in\mathbb{N}\cup\{0\}$. If $g\in T(G)$, then $\bigoplus\limits_{p\in P}p^{h_p(g)}T_p(G)=\mathbb{Z}g$.
\end{lemma}
\begin{proof}
	Let  $P_g=\{p\in P~|~h_p(g)=k_p<\infty\}$. Since $g\in T(G)$, the set $P_g$ is finite and consists of the prime divisors of $o(g)$. Let $p\in P_g$. Then the element $g$ can be represented in the following form
	\begin{equation}\label{eqq2.3}
		g=p^{k_p}s_pe_p+g',
	\end{equation}
	where $s_p\in\mathbb{Z}$, $gcd(p,s_p)=1$, $g'\in\bigoplus\limits_{q\in P_{_N}\setminus\{p\}}T_q(G)$. Let $m=o(g')$. Multiplying both sides of \eqref{eqq2.3} by $m$, we obtain
	\begin{equation}\label{eqq2.4}
		mg=p^{k_p}s_pme_p.
	\end{equation}
	Since $\gcd(p, m) = 1$, it follows that $x s_p m e_p = e_p$ for some $x \in \mathbb{Z}$. Multiplying both sides of \eqref{eqq2.4} by $x$, we obtain $p^{k_p} e_p = x m g \in \mathbb{Z} g$. Consequently, $p^{k_p} T_p(G) = (p^{k_p} \mathbb{Z}) e_p \subseteq \mathbb{Z} g$ for $p \in P_g$.

	Since $h_p(g)=\infty$ for each $p\in P\setminus P_g$, it follows that $\bigoplus\limits_{p\in P }p^{h_p(g)}T_p(G)\subseteq \mathbb{Z}g$. 		
	The reverse inclusion is obvious, so $\bigoplus\limits_{p\in P}p^{h_p(g)}T_p(G)=\mathbb{Z}g$.
\end{proof}

Next we will consider rings on groups in $\mathcal{QD}1$. To define a ring on a group, it is necessary to define a multiplication on it. A multiplication on a group $G$ is a homomorphism $\mu:~G\otimes G\rightarrow G$. This multiplication is often denoted by the sign $\times$, i.e. $\mu(g_1\otimes g_2)=g_1\times g_2$ for any $g_1,g_2\in G$. The ring on the group $G$, determined by the
multiplication $\times$, is denoted by $(G,\times)$.
On any group $G$, we can always define the multiplication $\mu: G \otimes G \rightarrow 0$, which is called to be trivial. If there are no multiplications on the group $G$ except the trivial multiplication, then $G$ is called a $nil$-group. Note that, according to \cite{K-N23},  every ring on a group $G \in \mathcal{QD}1$ is associative and commutative.

\begin{lemma}\label{lem2.2}
	Let $G\in\mathcal{RQD}1$, $\times$ be a multiplication on $G$ such that $G\times G\nsubseteq T(G)$, and let $g\in G\setminus T(G)$. Then $g\times G+\mathbb{Z}g=G(\text{\normalfont char}\,g).$
\end{lemma}
\begin{proof}
	It is easy to see that
	\begin{equation}\label{eqq2.7}
		g \times G+\mathbb{Z}g\subseteq G(\text{\normalfont char}\,g).
	\end{equation}
	We will prove the reverse inclusion. Let $\text{\normalfont cochar}\,G=\chi$, $E=\{e_p~|~p\in P_\chi\}$ be a $\Pi$-basis of the group $G$, $e=(e_p)_{p\in P_\chi}$. Let $b\in G(\text{\normalfont char}\,g)$, $P_0$ be a set that is $g$-defining, $e\times e$-defining and $b$-defining with respect to the $\Pi$-basis $E$; such a set exists due to \cite[Remark 2.1(2)]{K-N23}. Since $g,e\times e\notin T(G)$ by \cite[Remark 4.2]{K-N23}, it follows that these elements can be written as 
	\begin{equation}\label{eqq2.8}
		g=c_g\dfrac{r_1}{r_2}e_0+t_g,
	\end{equation}
	\begin{equation*}
		e\times e=c_\times\dfrac{m_1}{m_2}e_0+t_\times,
	\end{equation*}
	where $c_g=c(g)\neq 0$, $c_\times=c(e\times e)\neq 0$, $r_i,m_i$ are $P\setminus P_0$-integers $(i=1,2)$, $e_0=\pi_{P_0}(e)$, $t_g,t_\times\in \bigoplus\limits_{p\in P'}T_p(G)$, $P'=P_\chi\setminus P_0$. The element $b$ can be represented in the form 
	\begin{equation}\label{eqq2.9}
		b=c_g\dfrac{s_1}{s_2}e_0+t_b,
	\end{equation}
	where $s_1$ is a $(P\setminus P_0)\cup P_\infty$-integer,  $s_2$ is a $P\setminus P_0$-integer, $t_b\in\bigoplus\limits_{p\in P'}T_p(G)$.

	We denote $L=\Big\{c_g\dfrac{k_1}{k_2}e_0~|~ k_1\in\mathbb{Z}, k_2\;\text{is a}\; P\setminus P_0\text{-integer} \Big\}$ and show that $L\subseteq g\times G+\mathbb{Z}g$. Let $c_g\dfrac{k_1}{k_2}e_0\in L$ and let $n=o(t_g)$. Then $n$  is a $P'$-integer, and it means that $gcd(c_\times, nr_1k_2)=1$. Consequently, $c_\times x+nr_1k_2y=k_1$ for somes $x,y\in\mathbb{Z}$. Multiplying both sides of this equatility by $\dfrac{c_g}{k_2}$, we obtain 
	$$c_gc_\times\dfrac{x}{k_2}+c_gnr_1y=c_g\dfrac{k_1}{k_2},$$
	thus
	\begin{equation}\label{eqq2.10}
		c_g\dfrac{r_1}{r_2}  c_\times\dfrac{m_1}{m_2}  \dfrac{r_2m_2x}{r_1m_1k_2}+c_gnr_1y=c_g\dfrac{k_1}{k_2}.
	\end{equation}
	We set $z_1=r_2m_2x,z_2=r_1m_1k_2\in\mathbb{Z}$, then $z_2$ is a $P\setminus P_0$-integer. From \eqref{eqq2.10} we obtain 
	\begin{equation*}
		\begin{array}{ll}
			c_g\dfrac{k_1}{k_2}e_0&=\Big(c_g\dfrac{r_1}{r_2}e_0+t_g\Big ) \times \dfrac{z_1}{z_2}e_0+nyr_2\Big(c_g\dfrac{r_1}{r_2}e_0+t_g\Big )\\
			&=g\times \dfrac{z_1}{z_2}e_0+nyr_2g\in g\times G +\mathbb{Z}g.
		\end{array}
	\end{equation*}
	Therefore,
	\begin{equation}\label{eqq2.11}
		L\subseteq g\times G+\mathbb{Z}g.
	\end{equation}
	From \eqref{eqq2.11} we get $c_g\dfrac{r_1}{r_2}e_0\in g\times G+\mathbb{Z}g$. Since $g\in g\times G+\mathbb{Z}g$, it follows that 
	\begin{equation}\label{eqq2.11a}
		t_g\in g\times G+\mathbb{Z}g
	\end{equation} 
	by \eqref{eqq2.8}. Because $t_g\in T(G)$, we have $\bigoplus\limits_{p\in P}p^{h_p(t_g)}T_p(G)=\mathbb{Z}t_g\subseteq g\times G+\mathbb{Z}g$ by Lemma \ref{lem2.1} and \eqref{eqq2.11a}. Since $h_p(t_g)=h_p(g)$ for $p\in P'$ and $h_p(t_g)=\infty$ for $p\in P\setminus P'$, it follows that
	\begin{equation}\label{eqq2.12}
		\bigoplus\limits_{p\in P'}p^{h_p(g)}T_p(G)\subseteq g\times G+\mathbb{Z}g.
	\end{equation}
	From \eqref{eqq2.11} we obtain
	\begin{equation}\label{eqq2.13}
		c_g\dfrac{s_1}{s_2}e_0\in g\times G+\mathbb{Z}g.
	\end{equation}
	Since $b\in G(\text{\normalfont char}\,g)$, it follows that $t_b\in G(\text{\normalfont char}\,g)$. This means $t_b\in\bigoplus\limits_{p\in P'}p^{h_p(g)}T_p(G)$, hence
	\begin{equation}\label{eqq2.14}
		t_b\in g\times G+\mathbb{Z}g
	\end{equation}
	by \eqref{eqq2.12}. It follows from \eqref{eqq2.9}, \eqref{eqq2.13} and \eqref{eqq2.14} that $b\in g\times G+\mathbb{Z}g$, hence
	\begin{equation}\label{eqq2.15}
		G(\text{\normalfont char}\,g)\subseteq g\times G+\mathbb{Z}g.
	\end{equation}
	From \eqref{eqq2.7} and \eqref{eqq2.15} we conclude that $g\times G+\mathbb{Z}g=G(\text{\normalfont char}\,g).$
\end{proof}

Now we can describe the principal ideals of rings on groups in $\mathcal{QD}1$. Let $g\in G$, we denote by $(g)_\times$ the ideal of the ring $(G,\times)$ generated by $g$.

\begin{theorem}\label{the2.3}
	Let $G\in\mathcal{QD}1$, $\text{\normalfont cochar}\,G=\chi$, and let $(G,\times)$ be a ring, $g\in G$. 
	\begin{itemize}
		\item[1)] If $g\in T(G)=T$, then $(g)_\times =T(\text{\normalfont char}\,g)$. In addition, $T(\text{\normalfont char}\,g)=G(\text{\normalfont char}\,g)$ if and only if $G\in\mathcal{RQD}1$.
		\item[2)] If $g\notin T(G)$, $G\times G\nsubseteq T(G)$, then $(g)_\times =G(\text{\normalfont char}\,g)$.
		\item[3)] If $g\notin T(G)$, $G\times G\subseteq T(G)$, then $(g)_\times=\bigoplus\limits_{p\in P}p^{h_p(g\times e)}T_p(G)+\mathbb{Z}g$, where $\{e\}$ is a basis of $G$.
	\end{itemize}
\end{theorem}
\begin{proof}
	1) Let $g\in T(G)$. Then $T(\text{\normalfont char}\,g)=\bigoplus\limits_{p\in P}p^{h_p(g)}T_p(G)$. Since $g\in T(\text{\normalfont char}\,g)$ and $T(\text{\normalfont char}\,g)$ is a fully invariant subgroup of $G$, it follows that $(g)_\times\subseteq T(\text{\normalfont char}\,g)$. Since $g\in T(G)$, using Lemma \ref{lem2.1}, we obtain $T (\text{\normalfont char}\,g)=\mathbb{Z}g\subseteq (g)_\times$. 
	
	It is easy to see that $T(\text{\normalfont char}\,g)=G(\text{\normalfont char}\,g)$ if $G\in\mathcal{RQD}1$, $T(\text{\normalfont char}\,g)\neq G(\text{\normalfont char}\,g)$ if $G\in\mathcal{QD}1\setminus\mathcal{RQD}1$.

	2) Let $g \notin T(G)$, $G \times G \nsubseteq T(G)$. Since every multiplication on $G$ is associative and commutative \cite[Theorem 3.1(6)]{K-N23}, it follows that $(g)_\times = g \times G + \mathbb{Z}g$. If $G \in \mathcal{RQD}1$, then $(g)_\times = G(\text{\normalfont char}\, g)$ by Lemma \ref{lem2.2}.

	Let $G\in\mathcal{QD}1\setminus\mathcal{RQD}1$. Is is easy to see that  $(g)_\times\subseteq G(\text{\normalfont char}\,g)$. To prove the reverse inclusion, we represent the group $G$ in the form $G=\mathbb{Q}\oplus\mathbb{Z}_m$, where $m\in\mathbb{N}$. This decomposition is a decomposition of the ring $(G, \times)$ into the direct sum of ideals. Then $g = a + b$, where $a \in \mathbb{Q} \setminus \{0\}$, $b \in \mathbb{Z}_m$. 
	Since $G \times G \nsubseteq T(G)$, the ideal $\mathbb{Q}$ is isomorphic to the field of rational numbers and contains $a$. Consequently, $\mathbb{Q} \subseteq (g)_\times$, hence $\mathbb{Z}b \subseteq (g)_\times$. Thus, we obtain $G(\text{\normalfont char}\,g)=\mathbb{Q}\oplus \mathbb{Z}b\subseteq (g)_\times$.

	3) Let $g\notin T(G)$ and $G\times G\subseteq T(G)$. Then from \cite[Remark 4.2]{K-N23} it follows that there exist groups $A$ and $B$ such that $G=A\bigoplus B$ and $A\times G=0$, $B$ is a finite group. If $\{e\}$ is a basis of $G$, then $e=e_0+e_1$,  where $e_0\in A$ and $B=\mathbb{Z}e_1$, $o(e_1)<\infty$. Let $x\in G$. Then the elements $g$ and $x$ can be written in the form  $g=a+me_1$, $x=c+ne_1$, where $a,c\in A$, $m,n\in\mathbb{Z}$. We have $g\times x=mn(e_1\times e_1)$, $g\times e=m(e_1\times e_1)$. Thus  $h_p(g\times x)\geq h_p(g\times e)$ for any  $p\in P$. Therefore, $(g)_\times =g\times G+\mathbb{Z}g\subseteq \bigoplus\limits_{p\in P}p^{h_p(g\times e)}T_p(G)+\mathbb{Z}g$.
	
	To prove the reverse inclusion, we note that  $g\times e\in T(G)$. Therefore, according to Lemma \ref{lem2.1} we conclude that $\bigoplus\limits_{p\in P}p^{h_p(g\times e)}T_p(G)=\mathbb{Z}(g\times e)\subseteq (g)_\times$.
\end{proof}

\section{$AI$-rings and $FI$-rings on quotient divisible groups of rank 1}\label{sec3}
In this section we consider questions related to absolute ideals of groups in $\mathcal{QD}1$. In \cite{Fried64} a subgroup $F=\langle \text{\normalfont Im}\,\psi~|~\psi \in \text{\normalfont Hom}(G,\text{\normalfont End}\,G)\rangle$ of the endomorphism group $\text{\normalfont End}\,G$ was defined and it was proved that $F$ is an ideal of the endomorphism ring $E(G)$ of the group $G$. In addition, a fully invariant subgroup of $G$ is a group which is invariant with respect to $E(G)$, and an absolute ideal of $G$ is a subgroup which is invariant with resspect to $F$. So an absolute ideal is not necessarily a fully invariant subgroup. Recall that a group in which every absolute ideal is a fully invariant subgroup is called an $afi$-group. In a $nil$-group any subgroup is an absolute ideal, but some subgroups can not be fully invariant. For example, we consider a torsion-free group $G$ of rank 1 whose type $t$ is non-idempotent and $t(p)=\infty$ for some $p\in P$. Then $G$ is a $nil$-group. Let $g$ be any nonzero element of $G$ and let $\mathbb{Z}g$ be the cyclic group generated by $g$. Since $G$ is a $nil$-group, $\mathbb{Z}g$ is an absolute ideal of $G$ as noticed above, but $\mathbb{Z}g$ is not fully invariant in $G$ because $p^{-1}G\subseteq G$, but $p^{-1}g\notin \mathbb{Z}g$. Generalizing this example, we note that for any group $G$ every subgroup of the absolute annihilator $\text{\normalfont Ann}^*\,G$ of $G$ is an absolute ideal of $G$. In \cite{F12} it was shown that if $G$ is a torsion group, then $\text{\normalfont Ann}^*\,G$ coincides with the first Ulm subgroup $G^1=\bigcap\limits_{n\in\mathbb{N}}nG$ of $G$. This allowed in \cite{McL75a} to prove  that a separable torsion group $G$ is an $afi$-group if and only if $G^1$ is a cyclic group. More complicated examples of absolute ideals of a group $G$ that are not fully invariant subgroups of $G$ and are not contained in $\text{\normalfont Ann}^*\,G$ were given in \cite{K-Ph19}.


The first aim of this section is to prove that every quotient divisible group of rank 1 is an $RFI$-group, an $RAI$-group and an $afi$-group. To obtain these results, it is not necessary to describe the absolute ideals of the groups in $\mathcal{QD}1$; it is sufficient to use the relations between these classes proved in \cite{K-T22b}.
The intersection of the classes of $RFI$-groups, $RAI$-groups and $afi$-groups contains the class of $E$-groups, which were introduced by P. Schultz in \cite{Schultz88}. $E$-groups arise naturally in the theory of abelian groups when we consider groups isomorphic to their endomorphism groups. A group $G$ is called an $E$-group if $G$ is isomorphic to the endomorphism group $\text{\normalfont End}\,G$ and the endomorphism ring $E(G)$ is commutative. In \cite[Theorem 5.3]{K-T-T21}, it was shown that a group $G$ is an $E$-group if and only if every ring on $G$ is associative and $G$ admits the structure of a unital ring.

Let $\mathcal{E}$ be the class of all $E$-groups, $\mathcal{RFI}$ be the class of all $RFI$-groups, $\mathcal{RAI}$ be the class of all $RAI$-groups  and $\mathcal{AFI}$ be the class of all $afi$-groups. It was shown in \cite[Theorem 2.1]{K-T22b} that  $\mathcal{E}\subseteq \mathcal{RFI}=\mathcal{RAI}\cap \mathcal{AFI}.$

In Proposition \ref{pro3.1}, we will show that every group $G \in \mathcal{QD}1$ is an $E$-group. It follows that $G$ is an $RFI$-group, an $RAI$-group, and an $afi$-group.

\begin{proposition}\label{pro3.1}
	Every quotient divisible group of rank 1 is an $E$-group,
\end{proposition}
\begin{proof}
	Let $G$ be a group in $\mathcal{QD}1$ with the basis $\{e\}$. According to \cite[Theorem 3.2]{K-N23}, there exists a unique ring $(G,\cdot)$ in which $e\cdot e=e$. By \cite[Theorem 3.1]{K-N23} the ring $(G,\cdot)$ is a ring with the unity $e$.
	Since all multiplications on $G$ are associative by \cite[Theorem 3.1]{K-N23}, we obtain that $G$ is an $E$-group by \cite[Theorem 5.3 ]{K-T-T21}. 
\end{proof}

\begin{corollary}\label{cor3.2}
	Every quotient divisible group of rank 1 is an $RAI$-group, an $RFI$-group and an $afi$-group. \hfill$\square$\smallskip
\end{corollary}

Note that Corollary \ref{cor3.2} does not answer the question: which multiplications on groups in $\mathcal{QD}1$ determine AI-rings and $FI$-rings. To answer this question, we describe principal absolute ideals of groups $G\in\mathcal{QD}1$ (Theorem \ref{the3.3}). This allows us in Theorem \ref{the3.4} to describe the rings on the group $G$ in which all ideals are absolute ideals (respectively, fully invariant subgroups) of $G$, i.e. those rings on $G$ that are $AI$-rings (respectively, $FI$-rings). The description of principal absolute ideals of a group allows us to describe any of its absolute ideal, since any absolute ideal is the sum of principal absolute ideals.

\begin{theorem}\label{the3.3}
	Let $G\in\mathcal{QD}1$, $T(G)=T$, $g\in G$, $\langle g\rangle_{AI}$ be the absolute ideal of the group $G$ generated by the element $g$. Then $(g)_{AI}=G(\text{\normalfont char}\,g)$ if $g\notin T$; $(g)_{AI}=T(\text{\normalfont char}\,g)$ if $g\in T$.
\end{theorem}
\begin{proof}
	Let $g\notin T(G)$. Since $G\neq T(G)$  for any group  $G\in\mathcal{QD}1$, it follows from \cite[Theorem 3.1]{K-N23} that there exists a ring $(G,\times)$ such that $G\times G\nsubseteq T(G)$. By Theorem \ref{the2.3} we have $(g)_\times=G(\text{\normalfont char}\,g)$. Therefore, $G(\text{\normalfont char}\,g)\subseteq (g)_{AI}$. Since $G(\text{\normalfont char}\,g)$ is a fully invariant subgroup of the group $G$ and $(g)_{AI}$  is the smallest absolute ideal of the group $G$ containing $g$, we have $(g)_{AI}=G(\text{\normalfont char}\,g)$.
	
	If $g \in T(G)$, then by replacing the group $G(\text{\normalfont char}\, g)$ with $T(\text{\normalfont char}\, g)$ in the previous arguments, we obtain that $(g)_{AI} = T(\text{\normalfont char}\, g)$.
\end{proof}

\begin{theorem}\label{the3.4}
	Let $G\in\mathcal{QD}1$, $\{e\}$ be a basis of $G$.
	\begin{itemize}
		\item[1)] If $\text{\normalfont cochar}\, G = (\infty, \infty, \cdots, \infty, \cdots)$, then every ring on $G$ is an $FI$-ring (and, consequently, an $AI$-ring).
		\item[2)] If $\text{\normalfont cochar}\, G \neq (\infty, \infty, \cdots, \infty, \cdots)$ and $(G, \times)$ is a ring, then the following conditions are equivalent:
		\begin{enumerate}
			\item[a)] $(G,\times)$ is an $FI$-ring,
			\item[b)] $(G,\times)$ is an $AI$-ring,
			\item[c)] $e\times e\notin T(G)$,
			\item[d)] $G\times G\nsubseteq T(G)$.
		\end{enumerate}
	\end{itemize}
\end{theorem}
\begin{proof}
	1) If $\text{\normalfont cochar}\, G = (\infty, \infty, \cdots, \infty, \cdots)$, then $G$ is isomorphic to the additive group of integers. Consequently, every subgroup of the group $G$ is of the form $nG$ for some integer $n$, thus any ring on $G$ is an $FI$-ring.
	
	
	2) Let $\text{\normalfont cochar}\,G\neq (\infty,\infty,\cdots)$. The implication $a) \Rightarrow b)$ follows from the fact that any fully invariant subgroup of the group $G$ is its absolute ideal. 
	
	Now, let us show that $b) \Rightarrow c)$. Let $(G, \times)$ be an $AI$-ring. Assume that $e \times e \in T(G)$. Let us show that there exists a decomposition $G = A \oplus B$ such that
	\begin{equation}\label{eqq3.1}
		B\subseteq T(G),\quad A\times G=0,
	\end{equation}
	\begin{equation}\label{eqq3.2}
		pA=A\;\text{for some}\;p\in P.
	\end{equation}	
	If $G \in \mathcal{QD}1 \setminus \mathcal{RQD}1$, then the decomposition $G = \mathbb{Q} \oplus \mathbb{Z}_m\; (m \in \mathbb{Z})$ satisfies the conditions \eqref{eqq3.1} and \eqref{eqq3.2}.
	If $G \in \mathcal{RQD}1$ and $\text{\normalfont cochar}\, G = \chi$, then $P \setminus P_\infty(\chi) \neq \varnothing$ because $\chi \neq (\infty, \infty, \infty, \cdots)$. Let $P_\times = \{p \in P \mid \pi_p(e \times e) \neq 0\}$ (it is possible that $P_\times = \varnothing$ if $e \times e = 0$). Then, by \cite[Remark 4.2]{K-N23} there exists a non-empty finite subset $P_1$ of the set $P \setminus P_\infty(\chi)$ containing $P_\times$. Let $P_0 = P \setminus P_1$, $A = \pi_{P_0}(G)$, $B = \pi_{P_1}(G)$ (it is possible that $B = 0$ if $P_1 \cap P_\chi = \varnothing$). Then $A\subseteq G$, $B\subseteq G$ and the decomposition $G = A \oplus B$ satisfies the condition \eqref{eqq3.1}. Moreover, $pA = A$ for any $p \in P_1$.
	
	Thus, $G = A \oplus B$ and the groups $A, B$ satisfy the conditions \eqref{eqq3.1} and \eqref{eqq3.2}, so $e = e_0 + e_1$, where $e_0 \in A$, $e_1 \in B$. 
	Therefore, $\bigoplus\limits_{p \in P} p^{h_p(e_0 \times e)} T_p(G) = 0$. 
	By Theorem \ref{the2.3}, we get
	$$(e_0)_\times = \bigoplus\limits_{p \in P} p^{h_p(e_0 \times e)} T_p(G) + \mathbb{Z} e_0 = \mathbb{Z} e_0.$$
	From Theorem \ref{the3.3}, we obtain
	$$(e_0)_{AI} = G(\text{\normalfont char}\, e_0).$$
	Since $(G, \times)$ is an $AI$-ring, we have $(e_0)_\times = (e_0)_{AI}$ by \cite{K-T25}, thus $G(\text{\normalfont char}\, e_0) = \mathbb{Z} e_0$. Let $p \in P$ be such that $pA = A$. Since $\text{\normalfont char}\big(\dfrac{1}{p} e_0\big) = \text{\normalfont char}\, e_0$, we have $\dfrac{1}{p} e_0 \in G(\text{\normalfont char}\, e_0)$, which implies $\dfrac{1}{p} e_0 = n e_0 \in \mathbb{Z} e_0$ for some $n \in \mathbb{Z}$. Therefore $n p = 1$, since $o(e_0) = \infty$. The resulting contradiction proves that $e \times e \notin T(G)$.

	The implication $c) \Rightarrow d)$ is obvious.

	Let us show that $d) \Rightarrow a)$. Suppose that $G \times G \nsubseteq T(G)$ and $g \in G$. By Theorem \ref{the2.3} we have $(g)_\times = G(\text{\normalfont char}\, g)$ or $(g)_\times = T(\text{\normalfont char}\, g)$, where $T = T(A)$. Therefore, $(g)_\times$ is a fully invariant subgroup of the group $G$ for any $g \in G$. Any ideal $K$ of the ring $(G, \times)$ can be represented as $K = \sum\limits_{g \in K} (g)_\times$, that means $K$ is also a fully invariant subgroup of the group $G$. Therefore, $(G,\times)$ is an $FI$-ring.
\end{proof}

In conclusion, we note that, according to \cite{K-N23}, if $G$ is a group in $\mathcal{QD}1$, then the group $\text{\normalfont Mult}\, G$ of all multiplications of $G$ is isomorphic to the group $G$. This isomorphism takes each multiplication $\times$ in $\text{\normalfont Mult}\, G$ to the element $e \times e$, where $\{e\}$ is a basis of $G$. Let $M_{_{NAI}}$ be the set of multiplications on $G$ that determine rings which are not $AI$-rings. By Theorem \ref{the3.4}, we can assert that $M_{_{NAI}}$ is a subgroup of the group $\text{\normalfont Mult}\, G$ and coincides with the torsion part of $\text{\normalfont Mult}\, G$. Furthermore, $M_{_{NAI}} \cong T(G)$.

\end{document}